\documentclass[12pt]{article}

\usepackage[english]{babel}
\usepackage{graphicx}
\usepackage{amsmath}
\usepackage{amsthm}
\usepackage{amsfonts}
\usepackage{amssymb}
\usepackage{mathabx}
\usepackage{epstopdf}
\usepackage{xcolor}
\usepackage{enumerate}
\usepackage{comment}
\usepackage{bm}
\usepackage{cite}
\theoremstyle{plain}
\newtheorem{theorem}{Theorem}[section]

\newtheorem{prop}[theorem]{Proposition}

\usepackage[colorlinks=true,citecolor=black,linkcolor=black,urlcolor=blue]{hyperref}
\usepackage {tikz}
\usetikzlibrary{automata,arrows,calc,positioning}
\usetikzlibrary{ decorations.pathreplacing}
\usetikzlibrary{positioning, decorations.text}
\usepackage{pdfpages}
\usepackage{pgf,tikz,pgfplots}
\pgfplotsset{compat=1.14}
\usepackage{mathrsfs}
\usepackage{mathtools}
\usetikzlibrary{arrows}

\newcommand{\Relchangep}[2]{\mbox{Rel}(#1;#2)}

\newcommand{\splitRelchangep}[4]{\mbox{splitRel}_{#2,#3}(#1;#4)}

\begin{document}
\definecolor{ududff}{rgb}{0.30196078431372547,0.30196078431372547,1}

\title{ On the Real Reliability Roots of Graphs}
\author{Jason I. Brown  
and 
Isaac McMullin  \\ Department of Mathematics and Statistics, Dalhousie University,\\Halifax, Canada\\ \\}

\maketitle



\begin{abstract}
Consider a connected graph $G$, and assume that every edge fails independently with probability $q$. The {\em (all-terminal) reliability polynomial} is the probability in $q$ that the spanning connected subgraph of operational edges is connected. In this paper we focus on the real roots of reliability polynomials ({\em reliability roots}). We prove that almost every graph has a nonreal reliability root, and that the reliability polynomials of graphs have roots dense on the interval $[\beta,0]$ where $\beta\approx-0.5707202942$.
\end{abstract}

\vspace{0.25in}
\noindent \textit{Proposed running head}: Real Reliability Roots of Graphs

%
 %

\section{Introduction} 
A multigraph $G=(V,E)$ consists of a finite vertex set $V$ and a finite edge multiset $E$ consisting of unordered pairs of elements of $V$ (the {\em order} and {\em size} of $G$ are, respectively, $|V|$ and $|E|$). A simple graph, or simply a {\em graph}, is one where $E$ is a set (that is, there are no multiple edges). 
For standard graph theory terminology we refer the reader to \cite{west}.

There are many ways to model a network's robustness against random failure. One of the most common ways is as follows. Suppose that every vertex of a multigraph $G$ (we assume from here on in that all multigraphs and graphs are always connected) is always ``operational'' but each edge independently fails with probability $q \in [0,1]$ (so each edge is independently operational with probability $p = 1-q$). The {\em (all-terminal) reliability} $\Relchangep{G}{q}$ of graph $G$ is the probability that for any two vertices $u$ and $v$, there is a path of operational edges joining them, that is, that the spanning subgraph of $G$ consisting of the operational edges is connected. The history of reliability, in both theoretical and applied settings, is old and well established (see, for example, \cite{colbook}). 

For a graph $G$, its reliability can be expressed as
\begin{eqnarray*}
    \Relchangep{G}{q}&=&\sum_{E'}(1-q)^{|E'|}q^{|E|-|E'|}
\end{eqnarray*}
where the sum is taken over all $E'$ such that the spanning subgraph of $G$ with edge set $E'$ is connected. We see immediately that reliability functions are always polynomials in $q$ (and $p$), and hence we can refer to {\em reliability polynomials}.

For many graph polynomials, what has attracted attention is the {\em shape} of the coefficients of the polynomials, with many results and conjectures centering on {\em unimodality} -- the observation that the coefficient sequence is nondecreasing then nonincreasing. Indeed the coefficients of various expansions of the reliability polynomial had all been conjectured to be unimodal (see \cite{colbook}), and only recently has this been confirmed for some sequences \cite{flogconcave, hlogconcave}, with proofs using deep techniques from algebraic geometry. 

A famous -- and beautiful! -- theorem due to Newton (see, for example \cite{comtet}) proves that if a polynomial with positive coefficients has all real roots (i.e. is {\em real-rooted}) then the coefficient sequence is unimodal. This observation has led many to explore the real (and complex) roots of various graph polynomials, including {\em chromatic polynomials}, {\em matching polynomials}, {\em independence polynomials}, {\em domination polynomials}, and indeed reliability polynomials (see, for example, \cite{chromaticpolynomials, matchingpolynomials1, independencepolynomials, dominationpolynomials1, dominationpolynomials2, brownunitdiscconj}).
And indeed, while some graph polynomials, such as matching polynomials, always have all real roots, most graph polynomials have real and nonreal roots, and characterizing when the roots are always real (i.e. the polynomial is real-rooted) has been of considerable interest.

Some families of graphs, such as trees and cycles (the graphs of respectively {\em corank} $d = m-n+1$ of $0$ and $1$) have all real reliability roots. For corank $d = 2$, there are graphs that have nonreal reliability roots. The $\Theta$-graph, consisting of two vertices $x$ and $y$ joined by internally disjoint paths of lengths $l_{1}, l_{2}, l_{3}$, with $1 \leq l_{1} \leq l_{2} \leq l_{3}$, has all real reliability roots if and only if $l_{1}\leq l_{2}+l_{3}-2\sqrt{l_{2}l_{3}}=(\sqrt{l_{3}}-\sqrt{l_{2}})^{2}$. Even the multigraph $B_k$ consisting of a bundle of $k$ edges between two vertices has nonreal reliability roots -- the reliability polynomial of $B_{k}$ is $1-q^k$, which has all real roots if and only if $k =1$ or $2$.

While it is not the case that every reliability root is real, in \cite{brownunitdiscconj} it was shown that for any multigraph $G$ there exists a subdivision $G'$ (and hence infinitely many subdivisions) such that the reliability roots of $G'$ are all real. Despite this, the question of how common real-rootedness is for reliability polynomials has never been explored.

The closure of the real roots of reliability polynomials of multigraphs is precisely $[-1,0] \cup \{1\}$ \cite{brownunitdiscconj}. However, the proof in \cite{brownunitdiscconj} uses multigraphs in an essential way, and so the corresponding problem for graphs seems much more subtle and difficult.

This paper will address two main questions:
\begin{enumerate}
    \item How common is real-rootedness for reliability polynomials of graphs?
    \item What is the closure of the real roots of reliability polynomials of graphs?
\end{enumerate} 


\section{Almost All Graphs Have a Nonreal Reliability Root}

In the context of graphs, we can consider the well-known model $\mathcal{G}(n,\rho)$ of Erd\H{o}s-R\'{e}nyi random graphs (see, for example, \cite{bollobas}) and ask how common real-rootedness of reliability polynomials is. One might suspect from the result on graph subdivisions that real-rootedness is fairly common, but we shall see that this is far from the truth.

To do so, we need to write a couple of different expansions of the reliability polynomial (see, for example, \cite{colbook} for this and the following discussion of the forms). If $G$ is a graph of order $n$ and size $m$ (i.e. $G$ is an $(n,m)$-graph) then setting $d = m-n+1$, the corank of $G$, the $F$-form of the reliability polynomial is given by
\begin{eqnarray*}
    \Relchangep{G}{q}&=&\sum_{i=0}^{d}F_{i}q^{i}(1-q)^{m-i}.
\end{eqnarray*}
Here $F_{i}$ is the number of subsets of $i$ edges from $G$ whose removal leaves the graph connected. The set of all such subsets of $E(G)$ forms a {\em simplicial complex}, that is, a collection of sets closed under subsets, called the {\em cographic matroid} of $G$. Any matroid is {\em partitionable} into intervals (see \cite{colbook} for a definition). If $H_{i}$ is the number of intervals in our partition of the cographic matroid whose smallest set is of size $i$, the $H$-form of the reliability polynomial is given by
\begin{eqnarray*}
    \Relchangep{G}{q}&=&(1-q)^{n-1}\sum_{i=0}^{d}H_{i}q^{i}.
\end{eqnarray*}
The connection between the $F$- and $H$-forms is
\begin{eqnarray}
    F_{k} &=& \sum_{r=0}^{k}H_{r}{d-r \choose k-r} \label{HtoFform}
\end{eqnarray}
for $k = 0,1,\ldots,d$. We call the vectors $<F_{0},F_{1},...,F_{m-n+1}>$ and $<H_{0},H_{1},...,H_{m-n+1}>$ the $F$-vector and the $H$-vector respectively for the graph $G$. It is well known that all of the $F_{i}$'s and $H_{i}$'s are positive integers \cite{colbook} and in fact both vectors are unimodal \cite{flogconcave, hlogconcave}. 

Some additional notation is needed. A set of edges is a {\em cutset} if the removal of those edges results in the graph becoming disconnected. Let $c_{i}$ be the number of edge cutsets of $G$ of size $i$. We begin with the following result. 

\begin{prop}\label{c2bound}
Let $G$ be a connected multigraph of order $n$, size $m$ and corank $d=m-n+1 \geq 2$. If $c_{1}=0$ and $c_{2}<\frac{m(n-1)}{2d}$, then the reliability of $G$ has a nonreal root.
\end{prop}


\begin{proof}
We know that we can write the reliability in the $H$-form
\begin{eqnarray*}
\Relchangep{G}{q} &=& (1-q)^{n-1}\big(H_{0}+H_{1}q+H_{2}q^{2}+...+H_{d}q^{d}\big).
\end{eqnarray*}
Since the $(1-q)^{n-1}$ here only adds roots of $q=1$, we will ignore it and we need only consider the roots of the {\em $H$-polynomial}
\begin{eqnarray*}
h(q) &=& \sum_{i=0}^{d}H_{i}q^{i}.
\end{eqnarray*}
While calculating the roots of a polynomial precisely is in general intractable, there is an algorithm, due to Sturm, that can be used to determine if a polynomial with real coefficients has all real roots. If $f_{1}$ and $f_{2}$ are polynomials in variable $q$, let $\mbox{rem}(f_{1},f_{2})$ be the remainder of $f_{1}$ when divided by $f_{2}$. The {\em Sturm sequence} of a nonzero polynomial $f = f(q)$ is a sequence of polynomials $f_{0}, f_{1}, ..., f_{k}$ defined as 
\begin{eqnarray*}
    f_{0} & = & f,\\
    f_{1} & = & f',\\
    f_{i} & = & -\mbox{rem}(f_{i-1},f_{i-2}) \mbox{ for $i\geq 2$.}
\end{eqnarray*}
As the polynomials $f_{i}$ strictly decrease in degree, we can terminate this process at the last non-zero polynomial, $f_{k}$. The Sturm sequence is then $f_{0}, f_{1}, ..., f_{k}$. Sturm's Theorem (see, for example, \cite{broHick2}) states that if $f$ has positive leading coefficient, then $f$ has all real roots if and only if all terms in the Sturm sequence have positive leading coefficient and there are no gaps in degree, that is, $\mbox{deg}(f_{i+1})=\mbox{deg}(f_{i})-1$ for $i=0,1,\ldots,k-1$.

While we could consider the Sturm sequence of the polynomial $h(q)$, the difficulty arises in that we know more about the coefficients of $H_i$ when $i$ is small rather than when $i$ is large. So instead of looking at the Sturm sequence of $h(q)$ directly, we instead will examine the Sturm sequence of $f = f(q) = q^{d}h(1/q)$, the reversal of the polynomial $h(q)$. Clearly, $h(q) = \sum_{i=0}^{d}H_iq^i$ has all real roots if and only if $f(q) = q^{d}h(1/q)$ does. As $d\geq 2$, the Sturm sequence of $f$ must have at least 2 terms. Now
\begin{eqnarray*}
f_{0}&=&H_{0}q^{d}+H_{1}q^{d-1}+H_{2}q^{d-2}+...+H_{d}\\
f_{1}&=&dH_{0}q^{d-1}+(d-1)H_{1}q^{d-2}+(d-2)H_{2}q^{d-3}+...+H_{d-1}\\
f_{2}&=& -\mbox{rem}(f_{0},f_{1}).
\end{eqnarray*}
From (\ref{HtoFform}) we can determine that $H_{0}=F_{0}=1$ and since $G$ has no cut edges (i.e. $c_{1}=0$), $F_{1} = m = dH_{0} + H_{1}$, so $H_1 = n-1$.
$H_{2}$ can be found by calculating $F_{2}$ in two different ways. First, $F_{2}={m\choose2}-c_{2}$. On the other hand, again from (\ref{HtoFform}) we have  
\begin{eqnarray*}
F_{2} &=& H_{0}{d\choose2}+H_{1}{d-1\choose1}+H_{2}{d-2\choose0}\\
      &=& H_{2}+(d-1)(n-1)+{d\choose2}.
\end{eqnarray*}
It follows that 
\begin{eqnarray}\label{H2C2calc}
H_{2}&=& {m\choose2}-c_{2}-(d-1)(n-1)-{d\choose2}.
\end{eqnarray}

With all this given, we can look at the third term $f_2$ of the Sturm sequence. We will only be calculating the leading coefficient of $f_{2}$ as that is all we will need. A calculation shows that dividing $f_1$ into $f_0$ yields
\begin{eqnarray*}
& & H_{0}q^{d}+H_{1}q^{d-1}+H_{2}q^{d-2}+...+H_{d}\\
&=&(dH_{0}q^{d-1}+(d-1)H_{1}q^{d-2}+(d-2)H_{2}q^{d-3}+...+H_{d-1})\cdot\bigg(\frac{1}{d}q+\frac{n-1}{d^{2}}\bigg)\\
& &+\bigg(\frac{2H_{2}}{d}-\frac{(n-1)^{2}(d-1)}{d^{2}}\bigg)q^{d-2}+...\text{ },
\end{eqnarray*}
so we find that, using (\ref{H2C2calc}),
\begin{eqnarray*}
f_{2}&=& \bigg(\frac{(n-1)^{2}(d-1)}{d^{2}}-\frac{2H_{2}}{d}\bigg)q^{d-2}+...\\
&=& \bigg(\frac{(2c_{2}-n+1)m-2c_{2}(n-1)}{d^{2}}\bigg)q^{d-2}+...\text{ }.
\end{eqnarray*}
Thus, if the reliability polynomial of $G$ had all real roots, it would be the case that $f_2$ is identically $0$ or $f_2$ must have positive leading coefficient, so that 
\begin{eqnarray*}
\frac{(2c_{2}-n+1)m-2c_{2}(n-1)}{d^{2}}\geq 0.
\end{eqnarray*}
However, this would imply that 
\begin{eqnarray*} 
c_{2}\geq\frac{m(n-1)}{2d},
\end{eqnarray*}
which is not the case. Therefore, it must be that the reliability polynomial of $G$ has a nonreal root.
\end{proof}
We are now ready to prove our main result.

\begin{theorem}\label{almostallnonreal}
    Almost every graph $G \in \mathcal{G}(n,\rho)$ has a nonreal reliability root.
\end{theorem}
\begin{proof}

It is well known that for every fixed integer $k$, almost every  graph $G \in \mathcal{G}(n,\rho)$ is at least $k$-edge-connected ($\lambda(G)\geq k$) \cite{bollobas}. Therefore, for almost every such graph $G$, $\lambda(G)> 2$ and so $c_{2}=0<\frac{m(n-1)}{2(m-n+1)}$. It follows from Proposition \ref{c2bound} that the reliability polynomial of almost every  graph $G \in \mathcal{G}(n,\rho)$ has a nonreal root.
\end{proof}

\section{Density of the Real Reliabity Roots of Graphs}\label{atreldensity}

It is known \cite{brownunitdiscconj} that all real reliability roots of multigraphs are bounded, and are indeed in $[-1,0)\cup\{1\}$. As well, in fact for {\em multigraphs}, the region is covered --  the closure of real reliability roots for multigraphs is in fact all of $[-1,0]\cup\{1\}$  (we point out that to get the closure result, multiple edges play an essential role in the construction). This contrasts with the complex roots, which are only conjectured to be bounded.

In terms of the closure of the real reliability roots of graphs, the question seems more subtle than for multigraphs. What is not clear is whether the set of reliability roots of multigraphs strictly contains the set of reliability roots of graphs. Perhaps surprisingly, in \cite{rationalroots} it was shown that indeed the two sets differ, in that $q=-1$, which is the reliability root of many multigraphs, is never the reliability root of a graph (it is not hard to see that given any multigraph $G$ of order at least $2$, replacing each edge by a bundle of $2k$ edges transforms the reliability polynomial by replacing $q$ by $q^{2k}$, and hence $q = -1$ becomes a root as $1$ is always a root of a reliability polynomial). It seems natural then to focus on graphs.

Computations have led us to conjecture that the closure of the real reliability roots of graphs is in fact $[-1,0]\cup\{1\}$ as well, but a proof has been elusive. We can, however, find a subinterval of  $[-1,0]$ in the closure. Let us define 

\begin{eqnarray*}
    \beta&=&\frac{\big(3(708)^{1/2}-\frac{629}{8}\big)^{1/3}}{12}-\frac{23}{48\big(3(708)^{1/2}-\frac{629}{8}\big)^{1/3}}-\frac{5}{24}
\end{eqnarray*}
($\beta \approx -0.5707202942$ will turn out to be a root of a specific polynomial).

We now prove the following.

\begin{theorem}\label{densitysimplegraphs}
    The roots of the reliability polynomials of graphs are dense on the interval $[\beta,0]$.
\end{theorem}
\begin{proof}
    To begin, we need two results from \cite{brownmol}. Let $G$ be a graph with specified vertices $s$ and $t$. The {\em split reliability polynomial} of $G_{s,t}$, $\splitRelchangep{G}{s}{t}{q}$, is the probability that $G$ is split into two components with $s$ in one component and $t$ in the other if every edge has probability $q$ of being down. For more information on split reliability polynomials, see \cite{thesis}.
    
    In addition, we define a gadget $H(u,v)$ as a connected graph $H$ with specified vertices $u$ and $v$ ($u\neq v$). If $G$ is another graph, an {\em edge substitution} of $H(u,v)$ into $G$ labeled $G[H(u,v)]$ is the graph formed by replacing each edge $\{a,b\}\in E(G)$ by a copy of $H$, identifying $u$ with $a$ and $v$ with $b$ (or $u$ with $b$ and $v$ with $a$, as this does not change the following result). Note that if $u$ and $v$ are nonadjacent in $H$, then $G[H(u,v)]$ is a graph even when $G$ is a multigraph.
    
    Let $r\neq1$ be a root of $\Relchangep{G}{q}$. Then from  \cite{brownmol}, any solution of the equation
    \begin{eqnarray}\label{molsprel}
        \splitRelchangep{H}{u}{v}{q}&=&\frac{r}{1-r} \cdot \Relchangep{H}{q}
    \end{eqnarray}
    is either a root of $\Relchangep{G[H(u,v)]}{q}$ or $\Relchangep{H}{q}$.
    
    We begin by showing that $[-1/2,0]$ is in the closure of the real reliability roots of graphs. 
    First note that
    \begin{eqnarray*}
        \Relchangep{C_{n}}{q}&=&(1-q)^{n-1}(1+(n-1)q).
    \end{eqnarray*}
    Let $r=-\frac{1}{n-1}$, a root of $\Relchangep{C_{n}}{q}$. For this choice of $r$, (\ref{molsprel}) bceomes:
    \begin{eqnarray*}
        \splitRelchangep{H}{u}{v}{q}+\frac{1}{n} \cdot \Relchangep{H}{q}=0.
    \end{eqnarray*}
Now we make the following observation. Suppose that $\sigma$ is a simple root of $\splitRelchangep{H}{u}{v}{q}$. Then $\splitRelchangep{H}{u}{v}{q}$ changes sign on either side of $\sigma$, that is, for fixed small $\varepsilon > 0$, $\splitRelchangep{H}{u}{v}{\sigma-\varepsilon}$ and $\splitRelchangep{H}{u}{v}{\sigma+\varepsilon}$ have opposite signs. Clearly, if we take $n$ large enough, then 
\[ \splitRelchangep{H}{u}{v}{\sigma-\varepsilon}+\frac{1}{n}\cdot \Relchangep{H}{\sigma-\varepsilon} \]
has the same sign as $\splitRelchangep{H}{u}{v}{\sigma-\varepsilon}$, and similarly 
\[ \splitRelchangep{H}{u}{v}{\sigma+\varepsilon}+\frac{1}{n}\cdot \Relchangep{H}{\sigma+\varepsilon} \]
has the same sign as $\splitRelchangep{H}{u}{v}{\sigma+\varepsilon}$.
By the Intermediate Value Theorem, there is a real root of $\splitRelchangep{H}{u}{v}{q}+\frac{1}{n}\cdot\Relchangep{H}{q}$ in $(\sigma-\varepsilon,\sigma+\varepsilon)$.
Therefore, by (\ref{molsprel}), if $\sigma$ is a simple root of $\splitRelchangep{H}{u}{v}{q}$, for large enough $n$ we can find $t \in (\sigma-\varepsilon,\sigma+\varepsilon)$ such that $t$ is a root of $\Relchangep{C_{n}[H(u,v)]}{q}$.


Let $\eta$ and $k$ be positive integers with $\eta \geq 3$ and $1 \leq k \leq \eta - 2$. Construct a graph $H = H_{\eta,k}$ by taking the graph $P_{\eta}$ (a path of order $\eta$) with endpoints $u$ and $v$ and replacing any $\eta-1-k$ of the edges with a copy of $K_{3}$ (see Figure \ref{pathofk2andk3}).
    \begin{figure}
        \centering
        \begin{tikzpicture}[line cap=round,line join=round,>=triangle 45,x=1cm,y=1cm]
        \clip(-1,-1) rectangle (10.5,1.5);
        \draw [line width=1pt] (0,0)-- (2,0);
        \draw [line width=1pt] (4,0)-- (6,0);
        \draw [line width=1pt] (6,0)-- (6.25,0);
        \draw [line width=1pt] (7.75,0)-- (8,0);
        \draw [line width=1pt] (2,0)-- (4,0);
        \draw [line width=1pt] (8,0)-- (10,0);
        \draw [line width=1pt] (2,0)-- (3,1);
        \draw [line width=1pt] (4,0)-- (3,1);
        \draw [line width=1pt] (8,0)-- (9,1);
        \draw [line width=1pt] (10,0)-- (9,1);
        \draw (-0.5,0.5) node[anchor=north west] {$u$};
        \draw (10,0.5) node[anchor=north west] {$v$};
        \begin{scriptsize}
        \draw [fill=black] (0,0) circle (3pt);
        \draw [fill=black] (2,0) circle (3pt);
        \draw [fill=black] (4,0) circle (3pt);
        \draw [fill=black] (6,0) circle (3pt);
        \draw [fill=black] (8,0) circle (3pt);
        \draw [fill=black] (10,0) circle (3pt);
        \draw [fill=black] (6.5,0) circle (1pt);
        \draw [fill=black] (7,0) circle (1pt);
        \draw [fill=black] (7.5,0) circle (1pt);
        \draw [fill=black] (3,1) circle (3pt);
        \draw [fill=black] (9,1) circle (3pt);
        \end{scriptsize}
        \end{tikzpicture}
        \caption{Graph $H$.}
        \label{pathofk2andk3}
    \end{figure}

     In order for $H$ to be split into two components with $u$ and $v$ in different components, either exactly one of the single edges is down or exactly one of the $K_{3}$ graphs is split into two components with the vertices on the path in different components (the split reliability of $K_{3}$). The probability of exactly one of the single edges being down is
\begin{eqnarray*}
     kq(1-q)^{k-1}((1-q)^{2}(1+2q))^{\eta-1-k}
\end{eqnarray*}
where we have $k$ choices for the inactive edge and $(1-q)^{2}(1+2q)$ is the reliability of $K_{3}$. The probability that the disconnection happens at a $K_{3}$ is
\begin{eqnarray*}
     (\eta-1-k)(2q^{2}(1-q))(1-q)^{k}((1-q)^{2}(1+2q))^{\eta-2-k}
\end{eqnarray*}
where we have $\eta-1-k$ choices for the $K_{3}$ that is down and $2q^{2}(1-q)$ is the split reliability of $K_{3}$ for any choice of $u$ and $v$. We thus find that
    \begin{eqnarray*}
        \splitRelchangep{H}{u}{v}{q}&=&kq(1-q)^{k-1}((1-q)^{2}(1+2q))^{\eta-1-k}\\
        & &+(\eta-1-k)(2q^{2}(1-q))(1-q)^{k}((1-q)^{2}(1+2q))^{\eta-2-k}\\
        &=&q(1-q)^{2\eta-k-3}(1+2q)^{\eta-k-2}(k+2(\eta-1)q).
    \end{eqnarray*}
Clearly $\splitRelchangep{H}{u}{v}{q}$ has a simple root $q=-\frac{k}{2(\eta-1)}$, so by choosing specific values for $k$ and $\eta$, we can, with the previous observation, have real reliability roots of graphs arbitrarily close to any rational number in $[-\frac{1}{2},0]$. 

We now use a different gadget to prove density on $[\beta,-\frac{1}{2}]$.
    Setting $s = r/(1-r)$, we can rewrite (\ref{molsprel}) as 
    \begin{eqnarray*}\label{molsprels}
        \splitRelchangep{H}{u}{v}{q}&=&s \cdot \Relchangep{H}{q}.
    \end{eqnarray*}
    Since we know that the roots of all-terminal reliabilities of multigraphs are dense on $[-1,0]$, we see that the possible values for $s = r/(1-r)$ are dense on $[-\frac{1}{2},0]$. 
    
    For $H$ we will select $K_{4}-e$ where $e=\{u,v\}$ (see Figure \ref{K_4-e}). Calculations show that 
    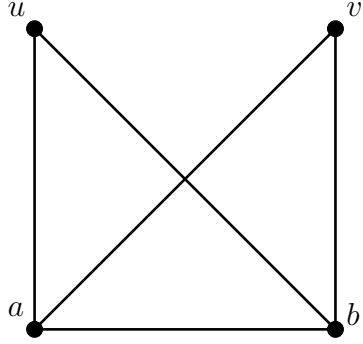
\begin{figure}
        \centering
        \begin{tikzpicture}[line cap=round,line join=round,>=triangle 45,x=1cm,y=1cm]
        \clip(0.5,0.5) rectangle (6,6);
        \draw [line width=1pt] (1,1)-- (5,1);
        \draw [line width=1pt] (1,1)-- (5,5);
        \draw [line width=1pt] (1,1)-- (1,5);
        \draw [line width=1pt] (5,1)-- (1,5);
        \draw [line width=1pt] (5,1)-- (5,5);
        \draw (0.5,5.5) node[anchor=north west] {$u$};
        \draw (5,5.5) node[anchor=north west] {$v$};
        \draw (0.5,1.5) node[anchor=north west] {$a$};
        \draw (5,1.5) node[anchor=north west] {$b$};
        \begin{scriptsize}
        \draw [fill=black] (1,1) circle (3pt);
        \draw [fill=black] (5,1) circle (3pt);
        \draw [fill=black] (1,5) circle (3pt);
        \draw [fill=black] (5,5) circle (3pt);
        \end{scriptsize}
        \end{tikzpicture}
        \caption{Our gadget $K_{4}-e$ with $e=\{u,v\}$}
        \label{K_4-e}
    \end{figure}
    \begin{eqnarray*}
        \Relchangep{K_{4}-e}{q}&=&(1-q)^{2}(-4q^{3}+q^{2}+2q+1)
     \end{eqnarray*}
and
    \begin{eqnarray*}
        \splitRelchangep{K_{4}-e}{u}{v}{q}&=&(1-q)^{2}(6q^{3}+2q^{2}),\\
    \end{eqnarray*}
so that
    \begin{eqnarray*}
    \splitRelchangep{K_{4}-e}{u}{v}{q}&=&s \cdot \Relchangep{K_{4}-e}{q}
    \end{eqnarray*}
if and only if 
    \begin{eqnarray*}
        s(-4q^{3}+q^{2}+2q+1)-6q^{3}-2q^{2} & = & 0.
    \end{eqnarray*}
    Solving the above equation in terms of $q$ via a Computer Algebra System gives us a solution:
    \begin{eqnarray*}
        q&=&\bigg((s-2)\big((12s+18)\sqrt{3}\sqrt{112s^{4}+256s^{3}+143s^{2}-8s}+253s^{3}+624s^{2}+390s-8\big)^{1/3}\\
        & &+\big(12s\sqrt{3}\sqrt{112s^{4}+256s^{3}+143s^{2}-8s}+18\sqrt{3}\sqrt{112s^{4}+256s^{3}+143s^{2}-8s}+\\
        & &253s^{3}+624s^{2}+390s-8\big)^{2/3}\bigg)\big/\\
        & &\bigg(\big((12s+18)\sqrt{3}\sqrt{112s^{4}+256s^{3}+143s^{2}-8s}+253s^{3}+624s^{2}+390s-8\big)^{1/3}+\\
        & &(12s+18)\bigg).
    \end{eqnarray*}
    The above equation is continuous for $s\in[-\frac{1}{2},-0.15]$ (there is a discontinuity to the right of $s=-0.15$) and is equal to $\beta$ when $s=-\frac{1}{2}$. The plot of the equation is shown in Figure \ref{plot_for_roots_to_-0.57}.
    \begin{figure}
        \centering
        \includegraphics[width=3.0in]{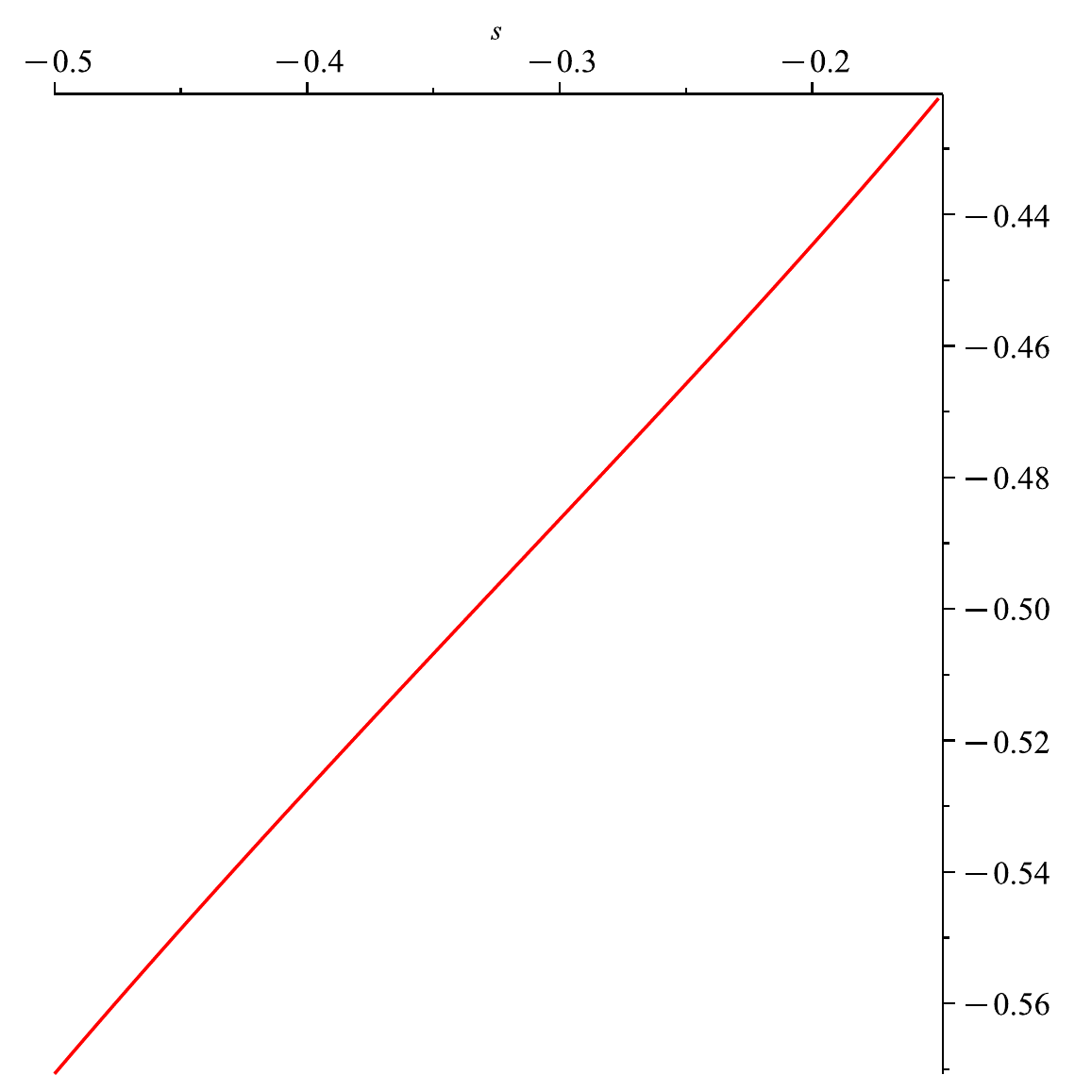}
        \caption{The plot of the possible reliability roots of a graph using a $K_{4}-e$ gadget replacement for $s\in[-\frac{1}{2},-0.15]$.}
        \label{plot_for_roots_to_-0.57}
    \end{figure}
    Since possible values for $s$ are dense in $[-\frac{1}{2},0]$, and the image of a dense set under a continuous map is dense in its image, it must be the case that the closure of the reliability roots for graphs contains the interval $[\beta,0]$.  
\end{proof}
\section{Open Problems}
While we were able to show that almost every graph has a nonreal reliability root, we can ask if almost every graph has a real reliability root (ignoring the obvious root $q=1$). Using the $H$-form of the reliability polynomial
\begin{eqnarray*}
    \Relchangep{G}{q}&=&(1-q)^{n-1}\sum_{i=0}^{d}H_{i}q^{i},
\end{eqnarray*}
we can also ask for which graphs $\Relchangep{G}{q}/(1-q)^{n-1}$ have 0 or 1 real roots. Since we know $H_{0} = 1 > 0$ and $H_d > 0 $,  if the graph's corank is odd, then we must have at least 1 real root.
By examining graphs for their reliability polynomials, we see that a majority of graphs of order up to 7 have no or exactly 1 real root. Is it perhaps the case in general that almost every graph has exactly 1 or 0 real roots?


In Section \ref{atreldensity} we provided an interval of width more than $1/2$ contained in the closure of the real reliability roots of graphs, but we conjecture that indeed the region should be all of $[-1,0] \cup \{1\}$. We already know that $-1$ is not a reliability root of a graph \cite{rationalroots}, but perhaps $-1$ is in the closure of the roots of reliability polynomials of graphs. Calculations show that the reliability polynomials of complete graphs have real roots approaching $-1$, but a closed form of $\Relchangep{K_{n}}{q}$ is not known, the observation remains tantalizingly out of reach at the moment.

\section*{Acknowledgements}  
 
J. Brown acknowledges research support from the Natural Sciences and Engineering Research Council of Canada (NSERC), grant RGPIN 2024-03846.

\bibliographystyle{plain}

\end{document}